\documentclass[11pt]{article}
\normalsize
\usepackage{amsfonts}
\usepackage{amsmath}
\usepackage{amssymb}

\newtheorem{theorem}{Theorem}
\newtheorem{prop}[theorem]{Proposition}
\newtheorem{definition}{Definition}
\newtheorem{lemma}[theorem]{Lemma}
\newtheorem{cor}[theorem]{Corollary}
\newtheorem{remark}{Remark}
\newtheorem{example}{Example}
\newtheorem{conj}{Conjecture}

\newenvironment{proof}{\noindent {\em Proof.}}{\hspace*{\fill} $\Box $\newline}

\newcommand{\F}{\mathbb{F}}

\newcommand{\Z}{\mathbb{Z}}
\newcommand{\GL}{\mbox{\rm GL}}

\newcommand{\Aut}{\mbox{\rm Aut}}
\newcommand{\wt}{\mbox{\rm wt}}

\title{ An Algorithm for Classification of Binary Self-Dual Codes}
\author{Stefka Bouyuklieva,\\
Faculty of Mathematics and Informatics,\\ Veliko Tarnovo University, Bulgaria,\\
Iliya Bouyukliev,\\ Institute of Mathematics and Informatics,\\
Bulgarian Academy of Sciences, Veliko Tarnovo, Bulgaria}
\date{}

\begin{document}
\maketitle

\textsl{We dedicate this research to our teacher Stefan Dodunekov
on the occasion of his 65th birthday.}

\begin{abstract}
An efficient algorithm for classification of binary self-dual
codes is presented. As an application, a complete classification
of the self-dual codes of length 38 is given. 
\end{abstract}

\section{Introduction}

The self-dual codes form one of the important classes of linear
codes because of their rich algebraic structure and their close
connections with other combinatorial configurations like block
designs, lattices, graphs, etc.

The classification of self-dual codes began in the seventies in
the work of Vera Pless \cite{PlessSO}, where she classified the
binary self-dual codes of length $n\leq 20$. The method used in
the beginning remained essentially the same throughout the
succeeding classifications. This is a recursive classification
which proceeds from smaller to larger length and codes are
classified up to equivalence. The process begins with the formula
for the number of all self-dual codes of length $n$ called a mass
formula. The number of the self-dual binary codes of even length
$n$ is $N(n)=\prod_{i=1}^{n/2-1} (2^i + 1)$.

Throughout this paper all codes are assumed to be binary. Two
binary codes are called \emph{equivalent} if one can be obtained
from the other by a permutation of coordinates. The permutation
$\sigma\in S_n$ is an \emph{automorphism} of the code $C$, if
$C=\sigma(C)$ and the set of all automorphisms of $C$ forms a
group called the \emph{automorphism group} of $C$, which is
denoted by $\Aut(C)$ in this paper. If $C$ has length $n$, then
the number of codes equivalent to $C$ is $n!/ |\Aut(C)|$. To
classify self-dual codes of length $n$, it is necessary to find
inequivalent self-dual codes $C_1,\dots, C_r$ so that the
following mass formula holds:
\begin{equation}\label{mass}
\sum_{i=1}^r\frac{n!}{|\Aut(C_i)|} = N(n).
\end{equation}

In the survey \cite{Huffman05} Huffman summarized the
classification of all binary self-dual codes of length $n\le 36$.
As the complete classifications for lengths 34 and 36 was not done
at that time, we present here a new version
of Table 1 from \cite{Huffman05}, 
but instead of the question marks 
we put the
correct numbers. Moreover, we extend the table with two more
lengths - 38 and 40. The number of all inequivalent singly-even
(Type I) and doubly-even (Type II) codes is denoted by $\sharp_I$
and $\sharp_{II}$, respectively. In the table $d_{max,I}$
($d_{max,II}$) is the biggest possible minimum distance for which
Type I (Type II) codes with the given length exist, and
$\sharp_{max,I}$ ($\sharp_{max,II}$) is their number. 

The classification of the self-dual codes of length 34 was given
in \cite{SD34}. Using the self-dual $[34,17,6]$ codes, Melchor and
Gaborit classified the optimal self-dual codes of length 36,
namely the $[36,18,8]$ codes \cite{n=36}. Recently, Harada and
Munemasa in \cite{Harada36} completed the classification of the
self-dual codes of length 36 and created a database of self-dual
codes \cite{database}. The doubly-even self-dual codes of length
40 were also classified by Betsumiya, Harada and Munemasa
\cite{BHM40}. Moreover, using these codes, they classified the
optimal self-dual $[38,19,8]$ codes. A classification of extremal
self-dual codes of length 38 was very recently obtained also in
\cite{AGKSS38} by somewhat different techniques.

Actually, the construction itself is easy, the equivalence test is
the difficult part of the classification. Because of that, for
larger lengths the recursive constructions are used preferably
heuristic, for building examples for codes with some properties.
There are also many partial classifications, namely
classifications of self-dual codes with special properties, for
example self-dual codes invariant under a given permutation, or
self-dual codes connected with combinatorial designs with given
parameters. These classifications are not recursive but they use
codes with smaller lengths, that is why the full classification is
very important in these cases, too.

There exist several methods to construct self-dual codes of length
$n+2$ from self-dual codes of length $n$. In \cite{AGKSS38}, the
authors describe three such methods recalling them \emph{the
recursive construction, the building-up construction and the
Harada-Munemasa construction}. The first one (recursive
construction) gives all inequivalent $[n+2,n/2+1,d+2]$ codes
starting from the self-dual $[n,n/2,d]$ codes \cite{n=36}. The
other two constructions give all self-dual codes of length $n+2$
from the self-dual codes of length $n$. Our construction is
similar to the Harada-Munemasa construction but we offer a better
way to deal with equivalence classes.

In this paper, we present an algorithm for generating binary
self-dual codes which gives as output exactly one representative
of every equivalence class. To develop this algorithm, we use an
approach introduced by Brendan McKay known as \emph{isomorph-free
exhaustive generation} \cite{McK}. The constructive part of the
algorithm is not different from the other recursive constructions
for self-dual codes, but to take only one representative of any
equivalence class, we use a completely different manner. This
approach changes extremely the speed of generation of the
inequivalent codes of a given length. Its special feature is that
practically there is not equivalence test for the objects.

As a result, the classification of all binary self-dual codes of
length 38 is presented. The number of these codes is given in the
following theorem and also listed in Table \ref{Table:F2}.

\begin{theorem}
There are 38 682 183 inequivalent self-dual codes of length $38$.
\end{theorem}


In Section \ref{Sect:Method} we present the theoretical
foundations of the used construction method. 
The aim is to
obtain all inequivalent self-dual codes of dimension $k$ (length
$2k$) on the base of the inequivalent self-dual codes of dimension
$k-1$ (length $2k-2$).

In Section \ref{Sect:Algorithm} we describe the used algorithm.
Codes which are equivalent belong to the same equivalence class.
Every code can serve as a representative for its equivalence
class. We use the concept for a canonical representative, selected
on the base of some specific conditions. This canonical
representative is intended to make easily a distinction between
the equivalence classes.

In Section \ref{Sect:Results} we give the results with some more
information about the obtained codes.

\begin{table}[htb]\rm
\begin{center}
\caption{Binary self-dual codes of length $n\leq 40$
\cite{Huffman05}} \vspace*{0.2in} \label{Table:F2}
\begin{tabular}{c|c|c|c|c|c|c|l}
\noalign{\hrule height1pt}
~~$n$~~&~~$\sharp_I$~~&~~$\sharp_{II}$~~&$d_{max,I}$&$\sharp_{max,I}$&
$d_{max,II}$&$\sharp_{max,II}$&Reference\\
\hline
2 &  1 &  &$2$& 1&     & &\cite{PlessSO} \\
4 &  1 &  &$2$& 1&     & &\cite{PlessSO} \\
6 &  1 &  &$2$& 1&     & &\cite{PlessSO} \\
8 &  1 &1 &$2$& 1&$4$&1&\cite{PlessSO} \\
10&  2 &  &$2$& 2&     & &\cite{PlessSO} \\
12&  3 &  &$4$& 1&     & &\cite{PlessSO} \\
14&  4 &  &$4$& 1&     & &\cite{PlessSO} \\
16& 5  &2 &$4$& 1&$4$&2&\cite{PlessSO} \\
18& 9  &  &$4$& 2&     & &\cite{PlessSO} \\
20& 16 &  &$4$& 7&     & &\cite{PlessSO} \\
22&25  &  &$6$& 1&     & &\cite{PlessSloane} \\
24&  46&9 &$6$& 1&$8$&1&\cite{PlessSloane} \\
26& 103&  &$6$& 1&     & &\cite{CP1,CPS:revised,PlessChildren} \\
28& 261&  &$6$& 3&     & &\cite{CP1,CPS:revised,PlessChildren} \\
30& 731&  &$6$&13&     & &\cite{CP1,CPS:revised,PlessChildren} \\
32&3210&85&$8$& 3&$8$&5&\cite{BilRees,CP1,C-S} \\
34& 24147  &  &$6$&938&    & &\cite{SD34} \\
36& 519492  &  &$8$&41&     & &\cite{Harada36,n=36} \\
38& 38682183  &  &$8$&2744&     & &this paper, \cite{AGKSS38,BHM40}\\
40& ?  & 94343  &$8$&?&   $8$  &16470 &\cite{BHM40} \\
  \noalign{\hrule height1pt}
\end{tabular}
\end{center}
\end{table}

\section{The Construction Method}
\label{Sect:Method}

In this section we describe the theoretical foundations of the
used construction method.
The aim is to obtain all inequivalent self-dual codes of length
$n=2k$ and dimension $k$ on the base of the inequivalent self-dual
codes of dimension $k-1$. Proposition \ref{thm:shortened} gives us
the possibility to develop a classification algorithm. Actually,
this proposition presents a well known property of the self-dual
codes but usually it is stated in the case $d\ge 4$ (see for
example \cite{Harada36}). Here we give a variant for all $d$ but
for that we need the following lemma. We say that the coordinates
$i_1,\dots,i_s$, $1\le i_1<\cdots<i_s\le n$, of the code $C$ are
equal if $x_{i_1}=x_{i_2}=\cdots=x_{i_s}$ for each codeword
$x=(x_1,x_2,\dots,x_n)\in C$.

\begin{lemma}\label{Lemma:prop}
If $C$ is a binary self-dual code then $C$ does not have more than
two equal coordinates. In other words, if
$x_{i_1}=x_{i_2}=\cdots=x_{i_s}$ for each codeword
$x=(x_1,x_2,\dots,x_n)\in C$ then $s\le 2$.
\end{lemma}

\begin{proof}
If $n=2$ then $C=i_2=\{00,11\}$ and $C$ has only two coordinates.
Let $n\ge 4$. Since the dimension of $C$ is $n/2\ge 2$, not all
coordinates are equal and $s<n$. Suppose that $s\ge 3$, so
$x_{i_1}=x_{i_2}=x_{i_3}$ for all codewords $x\in C$. In such a
case the vector $y=(y_1,\dots,y_n)\in\F_2^n$ of weight 2 with
$y_{i_1}=y_{i_2}=1$ will belong to $C^\perp=C$. In this way we
obtain a codeword for which $y_{i_1}=y_{i_2}=1$ but $y_{i_3}=0$ --
a contradiction. Hence $s\le 2$.
\end{proof}

Any self-dual $[n,n/2,2]$ code for $n>2$ is decomposable as
$i_2\oplus C_{n-2}$ where $C_{n-2}$ is a self-dual code of length
$n-2$. Furthermore, the number of the self-dual $[n,n/2,2]$ codes
is equal to the number of all self-dual $[n-2,n/2-1]$ codes
($n>2$). But even if a self-dual code of length $n>2$ has minimum
distance $d=2$, it has two coordinates which are not equal.

\begin{prop}\label{thm:shortened}
Let $C$ be a binary self-dual $[n=2k>2,k,d]$ code and
$C_0=\{(x_1,\dots,x_n)\in C, \ x_{n-1}=x_n\}$. If the last two
coordinates of $C$ are not equal and $C_1$ is the punctured code
of $C_0$ on the coordinate set $\{n-1,n\}$ then
$C_1$ is a self-dual $[n-2,k-1,d_1\ge d-2]$ code. 
\end{prop}

\begin{proof}
Let $n>2$ and the last two coordinates of the code $C$ are not
equal. This means that there is a codeword $y\in C$ such that
$y_{n-1}\neq y_n$. Therefore $C_0$ is a subcode of $C$ with
dimension $k-1$, and $(00\dots 011)\not\in C$. The punctured code
of $C_0$ on the coordinate set $\{n-1,n\}$ is
$$C_1=\{(x_1,x_2,\dots,x_{n-2}) \ | \
x=(x_1,x_2,\dots,x_{n-2},x_{n-1},x_n)\in C_0\}.$$ Since $(00\dots
011)\not\in C$, $\dim C_1 =\dim C_0=k-1$. Moreover, if
$x=(x_1,x_2,\dots,x_{n-2},a,a)$, $y=(y_1,y_2,\dots,y_{n-2},b,b)\in
C_0$ then $x_1y_1+x_2y_2+\cdots+x_{n-2}y_{n-2}+ab+ab=0$. Hence
$x_1y_1+x_2y_2+\cdots+x_{n-2}y_{n-2}=0$ for any two codewords
$(x_1,\dots,x_{n-2}), (y_1,\dots,y_{n-2})\in C_1$. It follows that
the code $C_1$ is self-orthogonal, and since its dimension is a
half of its length, it is self-dual.
\end{proof}

For our construction, we use the code $C_1$ of dimension $k-1$ to
obtain a self-dual code of dimension $k$ and length $2k$. We
describe this in the following corollary. Let $C_1$ be a binary
self-dual $[2k-2,k-1]$ code with a generator matrix $G_1$. Let us
extend $G_1$ with two equal columns $(a_1,a_2,\dots,a_{k-1})^T$
and consider the code $C_0$ generated by the matrix
\[
G_0=\left( \begin{array}{c|cc} &a_1&a_1\\
G_1&\vdots&\vdots\\
&a_{k-1}&a_{k-1}\\
\end{array}\right).
\]
We can choose the vector $(a_1,a_2,\dots,a_{k-1})$ such that
$(11\dots 11)\in C_0$.  Obviously, $C_0$ is a self-orthogonal
$[n=2k,k-1]$ code and its dual code is $C_0^\perp=\langle C_0,
(00\dots011), x\rangle$ where $x\in C_0^\perp\setminus \langle
C_0,(00\dots011)\rangle$ and $\wt(x)$ is even. Consider the last
two coordinates of the vector $x$. If $x_{n-1}=x_n$ then
$(x_1,x_2,\dots,x_{n-2})\in C_1^\perp=C_1$ and therefore
$x\in\langle C_0,(00\dots011)\rangle$. Hence $x_{n-1}\neq x_n$.
The next corollary follows immediately.

\begin{cor}\label{cor2}
The code $C=\langle C_0,x\rangle=C_0\cup (x+C_0)$ is a binary
self-dual $[2k,k]$ code.
\end{cor}

We will call the code $C$ a child type code (CTC) of $C_1$, and
$C_1$ - a parent type code (PTC) of $C$. Actually, we have two
possible self-dual codes obtained in this way from $C_1$, namely
$\langle C_0,x\rangle$ and $\langle C_0,(00\dots011)+x\rangle$,
but both codes are equivalent.

\begin{remark}\rm
To describe the search tree for our algorithm, we use the terms
CTC and PTC. In the literature on algorithms \cite{KO, McK} the
terms are parent and child, but as these terms are used in a
different way in some papers and chapters devoted to self-dual
codes \cite{HP}, we decided to set child type code and parent type
code.
\end{remark}

In this way all self-dual codes of length $n=2k$ can be
constructed on the base only on the inequivalent self-dual codes
of length $n-2$. Indeed, if we take two equivalent self-dual codes
$C_1\cong C_1'$ and use the same vector $(a_1,a_2,\dots,a_{k-1})$,
we can obtain equivalent self-dual codes of length $2k$ via the
described construction. Let $G_1$ be a generator matrix of $C_1$
and $P$ be a permutation matrix such that $G_1P$ generates the
code $C_1'$. Then the matrix
\[
\left( \begin{array}{c|cc}
xP&x_{n-1}&x_n\\
G_1P&a&a\\
\end{array}\right)
\]
generates a self-dual code equivalent to $C$, namely the code
$C'$. Furthermore, if $\pi\in S_{n-2}$ is the permutation
corresponding to the matrix $P$ then $\widehat{\pi}\in S_n$ sends
$C$ to $C'$ where $\widehat{\pi}(i)=\pi(i)$ for $1\le i\le n-2$,
and $\widehat{\pi}(i)=i$ for $i\in\{n-1,n\}$.

Let see now what happens if we take different vectors $a$ and $b$
of length $k-1$ and use them in the described construction for the
same self-dual $[2k-2,k-1]$ code $C_1$ with a generator matrix
$G_1$. Consider the elements of the automorphism group $\Aut(C_1)$
as permutation matrices of order $n-2$. To any permutation matrix
$P\in \Aut(C_1)$ we can correspond an invertible matrix $A_P\in
\GL(k-1,2)$ such that $G_1'=G_1P=A_PG_1$, since $G_1'$ is another
generator matrix of $C_1$. In this way we obtain a homomorphism $f
\ : \ \Aut(C_1) \longrightarrow  \GL(k-1,2)$. Consider the action
of $Im (f)$ on the set $\F_2^{k-1}$ defined by $A(x)=Ax^T$ for
every $x\in \F_2^{k-1}$.

\begin{theorem}\label{thm:punctured2aut} {\rm\cite{Harada36}}
The matrices $(G_1 \ a^T \ a^T)$ and $(G_1 \ b^T \ b^T)$  generate
equivalent codes if and only if the vectors $a$ and $b$ belong to
the same orbit under the action of $Im (f)$ on $\F_2^{k-1}$.
\end{theorem}

\begin{proof}
Let the matrices $(G_1 \ a^T \ a^T)$ and $(G_1 \ b^T \ b^T)$
generate the codes $C_0$ and $C_0'$, respectively, and
$a^T=A_Pb^T$, where $P\in\Aut(C_1)$. Then
\[
(G_1 \ a^T \ a^T)\left( \begin{array}{cc}
P&0\\
0&I_2\\
\end{array}
\right)=(G_1P \ a^T \ a^T)=(A_PG_1 \ A_Pb^T \ A_Pb^T)=A_P(G_1 \
b^T \ b^T).
\]
Since $A_P(G_1 \ b^T \ b^T)$ is another generator matrix of the
code $C_0'$, both codes are equivalent.

Conversely, let $C_0\cong C_0'$. It turns out that there is a
matrix $B\in\GL(k,2)$ and an $n\times n$ permutation matrix $P$
such that $(G_1 \ a^T \ a^T)=B(G_1 \ b^T \ b^T)P=BA_P(G_1 \ b^T \
b^T)$. Hence $BA_pG_1=G_1$ and therefore $BA_P$ defines an
automorphism of the code $C_1=\langle G_1\rangle$. Since $a^T=BA_P
b^T$, the vectors $a$ and $b$ belong to the same orbit under the
action of $Im (f)$ on $\F_2^{k-1}$.
\end{proof}

Let now $G=\left(%
\begin{array}{cc}
  x & 1 \ \ 0 \\
  G_1 & a^T \ a^T\\
\end{array}%
\right)$ be a generator matrix of the self-dual $[n,n/2,d]$ code
$C$. If $P\in \Aut(C_1)$ and $y=xP$ then
\[ G\left( \begin{array}{cc}
P&0\\
0&I_2\\
\end{array}
\right)=\left(%
\begin{array}{cc}
  x & 1 \ \ 0 \\
  G_1 & a^T \ a^T\\
\end{array}%
\right)\left( \begin{array}{cc}
P&0\\
0&I_2\\
\end{array}
\right)=\left(%
\begin{array}{cc}
  xP & 1 \ \ 0 \\
  G_1P & a^T \ a^T\\
\end{array}%
\right)\]
\[=\left(%
\begin{array}{cc}
  y & 1 \ \ \ \ 0 \\
  A_PG_1 & A_Pb^T \ A_Pb^T\\
\end{array}%
\right)=\left(%
\begin{array}{cc}
  1& 0\dots 0 \\
  0^T&A_P\\
\end{array}%
\right)\left(%
\begin{array}{cc}
  y & 1 \ \ 0 \\
  G_1 & b^T \ b^T\\
\end{array}%
\right)
\]

Hence the code $C$ is equivalent to the code generated by the matrix $G'=\left(%
\begin{array}{cc}
  y & 1 \ \ 0 \\
  G_1 & b^T \ b^T\\
\end{array}%
\right)$.

The following proposition reduces the number of the considered
cases. It is a particular case of Theorem 1 from \cite{Kim36-38}.

\begin{prop}\label{thm:Harada}
If $C$ is a binary self-dual $[n=2k>2,k,d]$ code then $C$ is
equivalent to a code with a generator matrix in the form
\begin{equation}\label{matrix1}
G=\left(
\begin{array}{cccc}
x_1 \dots x_{k-1}&00\dots 0&1&0\\
&&x_1&x_1\\
I_{k-1}&A&\vdots&\vdots\\
&&x_{k-1}&x_{k-1}\\
\end{array}
\right)
\end{equation}
and the matrix $(I_{k-1}|A)$ generates a self-dual $[n-2,k-1]$
code.
\end{prop}

\begin{proof}
If the last two coordinates of $C$ are equal then according to
Lemma \ref{Lemma:prop} the first and the last coordinates are not
equal and we can transpose them. So we can consider the case when
the last two coordinates of $C$ are not equal without loss of
generality.

Take a generator matrix in systematic form for the code $C_1$
obtained from $C$ by the construction in Proposition
\ref{thm:shortened}. Then
\[
\left(
\begin{array}{cccc}
&&x_1&x_1\\
I_{k-1}&A&\vdots&\vdots\\
&&x_{k-1}&x_{k-1}\\
\end{array}
\right)
\]
is a generator matrix of $C_0$ in systematic form. Since $C_0$ is
a self-orthogonal $[2k,k-1]$ code, $C_0^\perp=\langle C_0,
(00\dots011), y\rangle$ where $y\in C\setminus C_0$. Hence
$y_{n-1}\neq y_n$ and suppose $y_{n-1}=1$, $y_n=0$.

Consider the vector $x=(x_1,x_2,\dots,x_{k-1},00\dots010)$.
Obviously, $x\perp C_0$ and hence $x\in C_0^\perp$. But $x\not\in
C_0$, $x\not\in (00\dots011)+C_0$, $x\not\in (00\dots011)+y+C_0$
and therefore $x\in y+C_0\subset C$. It turns out that $x\in C$
and we can take a generator matrix for $C$ in the needed form.
\end{proof}

If we start with a self-dual $[2k-2,k-1]$ code $C_1$ with a
generator matrix in systematic form, for the construction in the
last theorem we have to use a vector $(x_1,x_2,\dots,x_{k-1})$ of
odd weight. Then we are sure that the all-ones vector $(11\dots
11)$ belongs to the subcode $C_0$ and the code $C$ generated by
the matrix (\ref{matrix1}) is a self-dual $[2k,k]$ code.

Theorem \ref{thm:punctured2aut} and Proposition \ref{thm:Harada}
are very important for our search. We use them to write a
recursive algorithm which gives us all self-dual $[n,n/2]$ codes
starting from the code $i_2$.

\section{The Algorithm}
\label{Sect:Algorithm}

In this section we present an algorithm for generating self-dual
codes of a given length $n$. We discuss the background and give
preliminary definitions and notations.

Let $\Omega_k$ be the set of all binary self-dual codes of
dimension $k$ (and length $2k$). We consider the action of the
symmetric group $S_{2k}$ on the set $\Omega_k$, $k=1,2,\dots$.
This action induces an equivalence relation in $\Omega_k$ as two
codes $C_1, C_2\in\Omega_k$ are equivalent ($C_1\cong C_2$) if
they belong to the same orbit. Hence the equivalence classes for
the defined relation are the orbits with respect to the action of
the symmetric group. According to Proposition \ref{thm:shortened},
if we take two nonproportional coordinates in any code belonging
to $\Omega_k$, $k\ge 2$, we can obtain a code from $\Omega_{k-1}$.
Conversely, if we take all codes from $\Omega_{k-1}$ and extend
them using Corollary \ref{cor2} in all possible ways, we will
obtain all codes from $\Omega_k$. The construction of the
self-dual codes of dimension $k$ using the codes from
$\Omega_{k-1}$ seems to be trivial, but actually this is a
difficult problem and the question is how to find only the
inequivalent codes. To do this, we develop a McKey type algorithm
for isomorph-free generation \cite{KO, McK}.

The algorithm is an exhaustive search over the set of codes
$\Omega=\Omega_1\cup\Omega_2\cup \dots\cup\Omega_k$ and this set
is our search space. The generation process is described by a
rooted tree (or forest). The nodes in the tree are objects from
the search space $\Omega$. From a self-dual code $A$ of length
$2k-2$ corresponding to the node $\overline{A}$, we obtain
self-dual codes of length $2k$ which are child type codes of $A$.
To construct the child type codes, we use a generator matrix $G_A$
of $A$ in systematic form and the binary vectors from $\F_2^{k-1}$
of odd weight. We denote the set of inequivalent child type codes
by $Child(A)$. The elements of $Child(A)$ correspond to the nodes
of the next level which are connected to $\overline{A}$ by edges.
To find only the inequivalent child type codes, we use Theorem
\ref{thm:punctured2aut}.
Practically, the rule $A\rightarrow Child(A)$ which juxtapose
child type codes to a code defines our search tree. The execution
of the algorithm can be considered as walking the search tree and
visiting all nodes through the edges. This can be done by a depth
first search. We present a pseudocode of the algorithm in Table
\ref{Algorithm}.


The following definitions and theorems help us to explain the
algorithm and to prove its correctness.

\begin{definition} A canonical representative map for the
action of the group $S_{2k}$ on the set $\Omega_k$ is a function
$\rho:\Omega_k\rightarrow\Omega_k$ that satisfies the following
two properties:

1. for all $X\in\Omega_k$ it holds that $\rho(X)\cong X$,

2. for all $X, Y\in\Omega_k$ it holds that $X\cong Y$ implies
$\rho(X) = \rho(Y)$.
\end{definition}

For a code $C\in\Omega_k$, the code $\rho(C)$ is the canonical
form of $C$ with respect to $\rho$. Analogously, $C$ is in
canonical form if $\rho(C)=C$. The code $\rho(C)$ is the canonical
representative of its equivalence class with respect to $\rho$.

We can take for a canonical representative of one equivalence
class a code which is more convenient for our purposes. Suppose
that $A$ is  the canonical representative of this class with
respect to a canonical representative map $\rho$. We can take
$B\cong A$ for a canonical representative for this class if we
change the canonical representative map in the following way:
$\rho'(X)=B$ if $X\cong B$, $\rho'(X)=\rho(X)$ if $X\not\cong B$.
According to Lemma \ref{Lemma:prop} a self-dual code does not have
more than two equal coordinates, hence if $k\ge 2$ we can take for
a canonical representative of any equivalence class a code for
which the last two coordinates are not equal.

\begin{example} \rm If we order the codewords in any code
lexicographically and then compare the codes according to a
lexicographical ordering of the columns, we will have one biggest
code in any equivalence class. We can take this code as a
canonical representative of its class.
\end{example}

\begin{definition} Let $C$ be a self-dual code of length $n$ and $\rho(C)\cong C$ be its canonical form.
A permutation $\phi_C\in S_n$ is called a canonical permutation of
the code $C$, if $\phi_C(C)=\rho(C)$.
\end{definition}

For a fixed canonical representative map $\rho$, the canonical
permutation of $C$ depends on the automorphism group of the code
and the permutations from the coset $\phi_C\Aut(C)$ can also be
canonical permutations since
$\phi_C\sigma(C)=\phi_C(\sigma(C))=\phi_C(C)=\rho(C)$ for any
$\sigma\in\Aut(C)$.

We obtain the automorphism group and a canonical form of a given
code $A$ using a modification of the algorithm presented in the
paper \cite{Iliya-aut}. This algorithm gives the order of the
group, a set of generating elements, and a canonical permutation.

The parent test can be defined in the following way. Suppose that
$A$ is a self-dual code of dimension $k-1$ and $B\in Child(A)$.
This means that $B$ is obtained from $A$ by the construction of
Proposition \ref{thm:Harada}. The parent test for $B$ depends on
the automorphism group of the code $B$.
Let $c_1$ and $c_2$ be the coordinate positions of the code $B$
for which $\phi_B(c_1)=2k, \phi_B(c_2)=2k-1$ where $\phi_B$ is the
canonical permutation of $B$.
The corresponding two columns of $B$ are not equal, because the
last two coordinates of the canonical representative are not
equal. We call the coordinates $c_1$ and $c_2$ special for the
code $B$ with respect to the canonical permutation $\phi_B$. If
there is an automorphism $\sigma$ of $B$ such that
$\{\sigma(c_1),\sigma(c_2)\}=\{2k-1,2k\}$ then the code $B$ passes
the parent test.
In such case we can change the canonical permutation, taking
$\phi_B\sigma^{-1}$ instead of $\phi_B$. Then the last two
coordinates of $B$ are special with respect to the new canonical
permutation. So a code passes the parent test, if there is a
canonical permutation $\phi$ for this code such that the last two
coordinates are special with respect to $\phi$.

Obviously, the canonical representative $B$ of one equivalence
class passes the parent test. If $B_1\cong B$ also passes the
parent test, then there is a permutation $\phi:B\rightarrow B_1$
such that $1\le\phi(i)\le n-2$ for $1\le i\le n-2$ and
$\{\phi(n-1),\phi(n)\}=\{n-1,n\}$, where $n$ is the length of the
codes. This means that the parent type codes of these two codes
are equivalent, too. So we have the following lemma.

\begin{lemma}\label{Lemma:parent}
If $B_1$ and $B_2$ are two equivalent self-dual $[2k,k]$ codes
which pass the parent test, their parent type codes are also
equivalent.
\end{lemma}

\begin{table}[!htb]\rm
\begin{center}
\caption{The main algorithm} \vspace*{0.2in} \label{Algorithm}
\begin{tabular}{l}
\noalign{\hrule height1pt}
Procedure Augmentation($A$: binary self-dual code);\\
begin\\
~~~If the dimension of $A$ is equal to $k$ then\\
~~~~~~begin\\
~~~~~~~~~$U_k:= U_k\cup \{A\}$;\\
~~~~~~~~~PRINT ($A, \sharp \Aut(A)$);\\
~~~~~~end;\\
~~~If the dimension of $A$ is less than $k$ then\\
~~~~~~begin\\
~~~~~~~~~find the set $Child(A)$ of all inequivalent child type
codes of
$A$;\\
~~~~~~~~~~~~~~~~~~~~~( using already known $\Aut(A)$;)\\
~~~~~~~~~For all codes $B$ from the set $Child(A)$ do the following:\\
~~~~~~~~~~~~if $B$ passes the parent test  then Augmentation($B$);\\
~~~~~~end;\\
end;\\
\\
Procedure Main;\\
Input:  ~~~$U_r$ -- nonempty set of binary self-dual $[2r,r]$ codes;\\
~~~~~~~~~~~~$k \ $ -- dimension of the output codes $(k>r)$;\\
Output: $U_k$  -- set of $[2k,k]$ binary self-dual codes;\\
begin\\
~~~$U_k:=\emptyset$ (the empty set);\\
~~~for all codes $A$ from $U_r$ do the following:\\
~~~begin\\
~~~~~~find the automorphism group of $A$;\\
~~~~~~Augmentation($A$);\\
~~~end;\\
end;\\
  \noalign{\hrule height1pt}
\end{tabular}
\end{center}
\end{table}



\begin{lemma}\label{Lemma:equ-parents}
Let $A_1$ and $A_2$ be two equivalent self-dual codes of dimension
$r$. Then for any child type code $B_1$ of $A_1$ which passes the
parent test, there is a child type code $B_2$ of $A_2$, equivalent
to $B_1$, such that $B_2$ also passes the parent test.
\end{lemma}

\begin{proof}  Let $G_1$ be a generator matrix of
$A_1$ in systematic form and $B_1$ be the code obtained from $A_1$
and the vector $a=(a_1,\dots,a_r)$ by the construction described
in Corollary \ref{cor2} and Proposition \ref{thm:Harada}. Let
$B_2$ be the code generated by the matrix $\pi(G_1)$ and the same
vector $a$, where $\pi\in S_{2r}$ and $\pi(A_1)=A_2$. Obviously,
$\pi(G_1)$ generates the code $\pi(A_1)=A_2$. Moreover, according
to Theorem \ref{thm:punctured2aut}, $B_2$ is equivalent to all
codes obtained by $A_2$ and the vectors from the orbit with
representative $a$ under the action of $\Aut(A_2)$ on $\F_2^r$.
Since the codes $B_1$ and $B_2$ are equivalent, they have the same
canonical representative $B=\rho(B_1)=\rho(B_2)$. The code $B_1$
passes the parent test and therefore there is a canonical
permutation $\phi_1:B_1\rightarrow B$ such that the last two
coordinates of $B_1$ are special with respect to $\phi_1$. We can
take for a canonical permutation of $B_2$ the permutation
$\phi_2=\phi_1\widehat{\pi}^{-1}$, since
$\phi_1\widehat{\pi}^{-1}(B_2)=\phi_1(B_1)=\rho(B_1)=\rho(B_2)$,
where $\widehat{\pi}\in S_{2r+2}$, $\widehat{\pi}(i)=\pi(i)$ for
$i\in\{1,2,\dots,2r\}$, $\widehat{\pi}(2r+1)=2r+1$,
$\widehat{\pi}(2r+2)=2r+2$. Then
$$\{\phi_2(2r+1),\phi_2(2r+2)\}=\{\phi_1(2r+1),\phi_1(2r+2)\}=\{2r+1,2r+2\},$$
hence the last two coordinates of $B_2$ are special with respect
to the canonical permutation $\phi_2$. It turns out that the code
$B_2$ also passes the parent test.
\end{proof}

\begin{theorem}
If the set $U_s$ consists of all inequivalent binary self-dual
$[2s,s]$ codes, then the set $U_{s+1}$ obtained by the algorithm
presented in Table \ref{Algorithm} consists of all inequivalent
self-dual $[2s+2,s+1]$ codes, $s\ge 1$.
\end{theorem}

\begin{proof}
We must show that the set $U_{s+1}$ filled out in Procedure
\textsc{Augmentation}, consists only of inequivalent codes, and
any binary self-dual $[2s+2,s+1]$ code is equivalent to a code in
the set $U_{s+1}$.

Suppose that the codes $B_1, B_2\in U_{s+1}$ are equivalent. Since
these two codes have passed the parent test, their parent type
codes are also equivalent according to Lemma \ref{Lemma:parent}.
These parent type codes are self-dual codes from the set $U_s$
which consists only in inequivalent codes. We have a contradiction
here and therefore the codes $B_1$ and $B_2$ cannot be equivalent.
It follows that $U_{s+1}$ consists of inequivalent codes.

Take now a binary self-dual code $C$ of dimension $s+1$ with a
canonical representative $B$. Hence $B$ is equivalent to $C$ and
$B$ passes the parent test. Since $U_s$ consists of all
inequivalent self-dual codes of dimension $s$, the parent type
code of $B$ is equivalent to a code $A\in U_s$. According to Lemma
\ref{Lemma:equ-parents}, there is a child type code $B_A$ of $A$,
equivalent to $B$, such that $B_A$ passes the parent test. Since
the codes $B$ and $B_A$ are equivalent, so are the codes $C$ and
$B_A$. In this way we find a code in $U_{s+1}$ which is equivalent
to $C$.
\end{proof}

Applying the algorithm recursively, we have the following
corollary.

\begin{cor}
If the set $U_r$ consists of all inequivalent binary self-dual
$[2r,r]$ codes, then the algorithm presented in Table
\ref{Algorithm} generates all inequivalent self-dual $[2k,k]$
codes, $r<k$.
\end{cor}

We can partition the set $U_r$ of all inequivalent binary
self-dual $[2r,r]$ codes into disjoint subsets $U_{r1},
U_{r2},\dots,U_{rs}$ and apply the algorithm to these subsets
independently. Denote by $U_{ki}$ the set of the inequivalent
self-dual $[2k,k]$ codes, obtained from $U_{ri}$, $i=1,2,\dots,s,$
via the described algorithm. Following the algorithm and the
theorems, we have

\begin{cor}
The union $U_{k1}\cup U_{k2}\cup\dots\cup U_{ks}$ consists of all
inequivalent binary self-dual $[2k,k]$ codes, and $U_{ki}\cap
U_{kj}=\emptyset$ for $i\neq j$.
\end{cor}

This corollary shows that we can divide the computations into
parts that need no mutual communication.

The most difficult part in the algorithm is the calculation of the
automorphism group and the canonical permutation of a given code.
That's why we try to avoid this part by using invariants.
Actually, we can do the parent test without knowing the
automorphism group of the code, using only appropriate invariants,
which is much faster.

\begin{definition}
Let $N=\{1,2,\dots,n\}$ be the set of the coordinates of the code
$C$. An invariant of the coordinates of $C$ is a function $f:
N\to\Z$ such that if $i$ and $j$ are in the same orbit with
respect to $\Aut(C)$ then $f(i)=f(j)$.
\end{definition}


The code $C$ and the invariant $f$ define a partition $\pi= \{
N_1,N_2,\dots,N_l\}$ of the coordinate set $N$,  such that
$N_i\cap N_j=\emptyset$ for $i\not =j$, $N=N_1\cup
N_2\cup\dots\cup N_l$, and two coordinates $i,j$ are in the same
subset of $N \iff  f(i)= f(j)$. So the subsets $N_i$ are unions of
orbits, therefore we call them pseudo-orbits. We can use the fact
that if we take two coordinates from two different subsets, for
example $s\in N_i$ and $t\in N_j$, $N_i\cap N_j=\emptyset$, they
belong to different orbits under the action of $\Aut(C)$ on the
coordinate set $N$. Moreover, using an invariant $f$, we can
define a new canonical representative and a new canonical
permutation of $C$.

Firstly, we introduce an ordering of the pseudo-orbits of $C$. We
say that $N_i\prec N_j$ for $i\neq j$, if: (1) $|N_i|<|N_j|$, or
(2) $|N_i|=|N_j|$ and $f(s)<f(t)$ for $s\in N_i$, $t\in N_j$. Then
we redefine the canonical representative of one equivalence class
in the following way:
\begin{enumerate}
\item If the smallest pseudo-orbit consists of only one
coordinate, we take for a representative a code in the equivalence
class for which this coordinate is the last one, and the $n-1$-th
coordinate is from the second smallest pseudo-orbit. The last two
coordinates of the code $C$ have been added according the
construction method described in the previous section. If none of
these two coordinates belongs to the smallest pseudo-orbit, $C$
does not pass the parent test. If one of them belongs to this
orbit, we check whether the other one belongs to the second
smallest pseudo-orbit. If no, $C$ does not pass the parent test,
but if yes, we should find the canonical representative to be sure
whether $C$ passes the test.

\item If the smallest pseudo-orbit $N_s$ contains three or more
coordinates, we again take for a representative a code in the
equivalence class for which the coordinates from the smallest
pseudo-orbit are the last coordinates. According to Lemma
\ref{Lemma:prop} not all coordinates in $N_s$ are equal, so we can
take two different coordinates in the end.

\item The smallest pseudo-orbit $N_s$ consists of exactly two
coordinates. If these two coordinates are different, we can take
them for the last two coordinates. But if they are equal, we are
looking for the next pseudo-orbit. There is only one self-dual
code such that its coordinates can be partitioned in subsets of
two elements such that both coordinates in each subset are equal,
and this is the code $i_2^k$, $n=2k$. But the automorphism group
of this code is $\Z_2^k\cdot S_k$ and it acts transitively on the
coordinates.

\end{enumerate}

The complexity of the algorithm mainly depends on two of its
steps. The first one is "\textit{find the set $Child(A)$ of all
inequivalent child type codes of} $A$".  
For this step, we have as input a set $L$ of generating elements
of the automorphism group of the code $A$, and the set $D$ of all
binary odd-weight vectors with length $\dim A$. The algorithm
splits the set of these vectors into orbits under the action of
the group $\Aut(A)$. Any orbit defines a child type code of $A$,
as different orbits give inequivalent child type codes. This step
is computationally cheap. The complexity is linear with respect to
the product $|L|\cdot |D|$, $|D|=2^{\dim A-1}$ (for details see
\cite{perm-group-book}).

In the step "\textit{if $B$ passes the parent test}", using a
given generator matrix of the code $B$ we have to calculate
invariants, and in some cases also canonical form and the
automorphism group $\Aut(B)$. Finding a canonical form and the
automorphism group is necessary when the used invariants are not
enough to prove whether the code $B$ pass or not the parent test.
If the code $B$ passes the parents test, the algorithm needs a set
of generators of $\Aut(B)$ for the next step (finding the child
type codes). For this step, we have to generate the set $M_w$ of
all codewords of weight $\le w$. The complexity here is
$O(\sum_{i=1}^{w/2}{\dim B\choose i})$. To determine the
invariants, we use the set $M_d$, where $d$ is the minimum weight
of the code $B$. We use mostly the invariants $f_1$ and $f_2$
defined as follows:
$$f_1:\{1,2,\dots,n\}\rightarrow
\{0,1\}, \ \ f_2:\{1,2,\dots,n\}\rightarrow \Z,$$ where $n$ is the
length of $B$. Moreover, if $s=\sum_{v\in M_d} v_i$ then
$f_1(i)=1$ if and only if $s$ is odd, and $f_2(i)=s$, $1\le i\le
n$, $v=(v_1,v_2,\dots,v_n)\in\F_2^n$. To calculate the values of
$f_1$, the algorithm needs only $|M_d|$ operations in the case of
bitwise presentation of the codewords. For $f_2$ the algorithm
uses $O(n|M_d|)$ operations. We use also a vector valued invariant
$f$ such that $f(i)=(f_1(i),f_2(i))$. If the code is not rejected
with the invariants, the algorithm generates the smallest set
$M_w$ of rank $\dim B$. The corresponding to this set binary
matrix is the input in the algorithm for a canonical form. For the
complexity of this type of algorithms see \cite{Iliya-aut, BM}.

The described algorithm is implemented in the program
\textsc{Gen-self-dual-bin} of the package \textsc{Self-dual-bin}
written in \textsc{Borland Delphi 6.0}.

\section{The Results}
\label{Sect:Results}

To find all inequivalent self-dual codes of length 38, we begin
from the set $U_{16}$ of all inequivalent self-dual codes of
length 32 and dimension 16. We partition $U_{16}$ into three
subsets and run them on three cores using a PC Intel i5 4 core
processor. The number of all codes considered in the program
(these are the inequivalent child type codes for each code in
$U_{16}\cup U_{17}\cup U_{18}$) is 2,338,260,952. For 151,016,675
of them, a canonical form is computed. The number of the obtained
inequivalent codes in the set $U_{19}$ is 38,682,183. The
calculations took about four days.

Generator matrices of all inequivalent self-dual codes of length
38 are saved in  three files in the form used in \cite{Bo_SJC07}.
A compressed version of these files is available on the web-page
\verb"http://www.moi.math.bas.bg/~iliya/". A file with additional
information which contains the number of codewords of weights 2,
4, 6, 8, and the order of the automorphism group for each code is
also available. This information allows to compute the expressions
in the mass formulas.


The possible weight enumerators of the self-dual codes of length
38 are given by the following formula
\begin{align*}
W(y)=&1+\alpha y^2+(13\alpha+\beta)y^4+(57+76\alpha+7\beta-
\gamma-4\delta)y^6+(228+260\alpha+17\beta-
\gamma+28\delta)y^8\\
 &+(560\alpha+ 7\beta +6\gamma-136\delta+1520)y^{10}+(728\alpha-43\beta+6\gamma +248\delta +10032)y^{12}\\
 &+(364\alpha- 77\beta -15\gamma+100\delta+37620)y^{14}+(85614-572\alpha-11\beta-15\gamma -764\delta)y^{16}\\
 &+(127072 - 1430\alpha+99\beta +20\gamma+528\delta)y^{18}+ \cdots +y^{38},
\end{align*}
where $\alpha,\beta,\gamma,\delta$ are integers. The numbers of
inequivalent codes, the numbers of different weight enumerators
and the numbers of different orders of the automorphism groups for
each minimum weight $d$ are given in Table \ref{Table:codes}.

\begin{table}[htb]\rm
\begin{center}
\caption{Numbers of inequivalent codes of length
38}\vspace*{0.2in} \label{Table:codes}
\begin{tabular}{c|cccc}
\noalign{\hrule height1pt} $d$&2&4&6&8\\
\hline  $\sharp$ codes&519492&27463982&10695965&2744\\ \hline
$\sharp$ weight enumerators&3504&7176&88&2\\ \hline $\sharp$
orders of
$\Aut(C)$&799&764&75&18\\
\noalign{\hrule height1pt}
\end{tabular}
\end{center}
\end{table}

The smallest order $\sharp \Aut_s$ and the largest order
$\sharp\Aut_l$ among the automorphism groups are listed in Table
\ref{s-l} for each minimum weight $d$.

\begin{table}[htb]\rm
\begin{center}
\caption{Orders of the automorphism groups}\vspace*{0.2in}
\label{s-l}
\begin{tabular}{c|cccc}
\noalign{\hrule height1pt} $d$&2&4&6&8\\
\hline  $\sharp\Aut_s$ &2&4&1&1\\ \hline
$\sharp\Aut_l$&$2^{19}\cdot 19!$&$2^{13}\cdot 3\cdot 7\cdot 16!$&1032192&504\\
\noalign{\hrule height1pt}
\end{tabular}
\end{center}
\end{table}

We give one more table with some additional results about the
automorphisms of the codes. We list in Table \ref{Table:aut} the
number of self-dual $[38,19,d]$ codes $C$ such that $p^k$ divides
the order of $\Aut(C)$ where $p$ is a prime and $k$ is a positive
integer.

For the verification of our results, we use the mass formula
(\ref{mass}) and also the following corollary from one lemma of
Thompson \cite{Thompson}:

\begin{theorem}\label{lemma:Thompson}{\rm \cite{Harada36}}
Let $n$ and $d$ be even positive integers and let $U$ be a family
of inequivalent self-dual codes of length $n$ and minimum weight
at most $d$. Then $U$ is a complete set of representatives for
equivalence classes of self-dual codes of length $n$ and minimum
weight at most $d$ if and only if
\begin{equation}\label{mass-d}
\sum_{C\in U}\frac{n!}{|\Aut(C)|}|\{x\in C\vert
\wt(x)=d\}|={n\choose d}\prod_{i=1}^{n/2-2}(2^i+1).
\end{equation}
\end{theorem}

For the constructed codes we obtain the same values of the left
and the right expressions in the formula (\ref{mass-d}) for
$n=38$, namely:
\begin{itemize}
\item[(d=2)] 19137697424578816915816164139573797711865999715625;
\item[(d=4)] 2009458229580775776160697234655248759745929970140625;
\item[(d=6)]
75153737786321014028410076576106303614497780883259375;
\item[(d=8)]
1331294783643400819931835642205311664028246404217737500;
\item[(all)] 27222898185745116523209337325140537285726884375 (formula (\ref{mass})).
\end{itemize}

\begin{remark}\rm
To calculate the sums in the mass-formulas (\ref{mass}) and
(\ref{mass-d}), we use the package \textsc{LONGNUM} of S. Kapralov
for calculations with large integers \cite{Kapralov}.
\end{remark}

\section{Conclusion}

The generation of all inequivalent binary self-dual codes of
length $n\ge 38$, using only standard computer algebra systems,
seems to be infeasible. That is why we use special algorithmic
techniques to surmount difficulties and to classify codes even
with PC's. In this work we describe the classification of the
self-dual codes of length 38.

Denote by $SD_k$ the number of all inequivalent self-dual codes of
dimension $k$. Obviously,  $$SD_k\ge
\displaystyle\frac{\prod_{i=1}^{k-1} (2^i + 1)}{(2k)!}.$$ Consider
the sequence $a_k=SD_k(2k)!/\prod_{i=1}^{k-1} (2^i + 1)$,
$k=1,2,\dots$.  Looking at the already known classifications of
binary self-dual codes, we can calculate the values of $a_k$ for
$k\le 19$. We list the integer part of these values in Table
\ref{Table:ak}. Moreover, if we count the number of the
inequivalent self-dual codes with a trivial automorphism group, we
see that there are no such codes for length $n\le 32$, but for the
larger lengths this number increases very fast. For example, more
than a quarter of all codes of length 38, namely 10140257
inequivalent codes, have a trivial automorphism group. So we have
the following conjecture.


\begin{conj}
The sequence $\{a_k, \ k=10,11,12,\dots\}$ is decreasing.
\end{conj}

\begin{table}[htb]\rm
\begin{center}
\caption{Values of $\lfloor a_k\rfloor$ for $k\le
19$}\vspace*{0.2in} \label{Table:ak}
\begin{tabular}{c|cccccccccc}
\noalign{\hrule height1pt} $k$&1&2&3&4&5&6&7&8&9&10\\
\hline  $\lfloor a_k\rfloor$&2&8&48&597&3162&18974&70836&
230631&353061&464937\\
\noalign{\hrule height1pt}
$k$&11&12&13&14&15&16&17&18&19&20\\
\hline  $\lfloor
a_k\rfloor$&327440&194067&57659&13482&2004&273&34&7
&2&?\\
\noalign{\hrule height1pt}
\end{tabular}
\end{center}
\end{table}

The doubly-even self-dual codes of length 40 have been classified
in \cite{BHM40}. 
A lower bound on the number of all self-dual codes of that length
is given by $\displaystyle\frac{\prod_{i=1}^{19} (2^i +
1)}{40!}>4,585,657,509$. According to Conjecture 1 and Table
\ref{Table:ak}, we have

\begin{conj} The number of the inequivalent binary self-dual
codes of length $40$ are at most $2\cdot
\displaystyle\frac{\prod_{i=1}^{19} (2^i + 1)}{40!}<
9,171,315,020$.
\end{conj}


\begin{table}[htb]\rm
\begin{center}
\caption{Number of codes $C$ for which $p^k$ divides the order of
their automorphism groups}\vspace*{0.2in} \label{Table:aut}
\begin{tabular}{|l|cccc|}
\noalign{\hrule height1pt} $p^k\setminus d$&2&4&6&8\\
\hline  2 &519492&27463982&557127&453\\
$2^2$&478434&27463982&89141&111\\
$2^3$&467048&17031875&23498&34\\
$2^4$&397699&16153273&8201&1\\
$2^5$&375614&9310617&3788&-\\
$2^6$&303984&8307428&1733&-\\
$2^7$&274730&4686394&846&-\\
$2^8$&212982&3854920&414&-\\
$2^9$&183525&2208800&220&-\\
$2^{10}$&140158&1629411&134&-\\
$2^{11}$&112968&970908&79&-\\
$2^{12}$&85784&641657&48&-\\
$2^{13}$&65569&387900&26&-\\
$2^{14}$&48579&237143&14&-\\
$2^{15}$&35399&140837&8&-\\
$2^{16}$&25470&83500&5&-\\
$2^{17}$&17908&49360&-&-\\
$2^{18}$&12729&29175&-&-\\
$2^{19}$&8838&17009&-&-\\
$2^{20}$&6161&9646&-&-\\
$2^{21}$&4162&5612&-&-\\
$2^{22}$& 2774&3228&-&-\\
$2^{23}$&1812&1754&-&-\\
$2^{24}$&1177&996&-&-\\
$2^{25}$&790&505&-&-\\
$2^{26}$&551&235&-&-\\
$2^{27}$&345&91&-&-\\
$2^{28}$&226&33&-&-\\
$2^{29}$&156&6&-&-\\
 \noalign{\hrule height1pt}
\end{tabular} \ \
\begin{tabular}{|l|cccc|}
\noalign{\hrule height1pt} $p^k\setminus d$&2&4&6&8\\
\hline  $2^{30}$&87&2&-&-\\
$2^{31}$&47&1&-&-\\
$2^{32}$&16&-&-&-\\
$2^{33}$&7&-&-&-\\
$2^{34}$&3&-&-&-\\
$2^{35}$&1&-&-&-\\
&&&&-\\
3 &132329&2743510&3916&85\\
$3^2$&45728&424433&185&7\\
$3^3$&17286&93437&26&1\\
$3^4$&6463&22451&3&-\\
$3^5$&1763&4315&-&-\\
$3^6$&256&366&-&-\\
$3^7$&51&52&-&-\\
$3^8$&8&5&-&-\\
5 &4209&16112&9&-\\
$5^2$&449&624&-&-\\
$5^3$&52&40&-&-\\
7 &4270&21981&14&5\\
$7^2$ &467&771&-&-\\
$7^3$ &83&83&-&-\\
$7^4$ &16&13&-&-\\
$7^5$ &1&2&-&-\\
11&47&29&-&-\\
13&17&7&-&-\\
17&4&-&-&-\\
19&1&-&2&1\\
23&6&4&-&-\\
31&2&2&-&-\\
 \noalign{\hrule height1pt}
\end{tabular}
\end{center}
\end{table}

\section*{Acknowledgements}


The main part of this research was done during the authors' visit
to Department of Algebra and Geometry at Magdeburg University,
Germany. The authors would like to thank Prof. Wolfgang Willems
for his hospitality and support. Stefka Bouyuklieva also thanks
the Alexander von Humboldt Foundation for the financial support.



\end{document}